\documentclass[12pt,oneside,english]{amsart}
\usepackage[T1]{fontenc}
\usepackage[latin9]{inputenc}
\usepackage[paperwidth=8.5in,paperheight=11in]{geometry}
\usepackage{amsthm}
\usepackage{amssymb}
\usepackage{graphicx}

\makeatletter
\theoremstyle{plain}
\newtheorem{thm}{\protect\theoremname}[section]
  \theoremstyle{plain}
  \newtheorem{conjecture}[thm]{\protect\conjecturename}
  \theoremstyle{definition}
  \newtheorem{defn}[thm]{\protect\definitionname}
  \theoremstyle{plain}
  \newtheorem{prop}[thm]{\protect\propositionname}
  \theoremstyle{plain}
  \newtheorem{cor}[thm]{\protect\corollaryname}
  \theoremstyle{remark}
  \newtheorem{rem}[thm]{\protect\remarkname}
  \theoremstyle{plain}
  \newtheorem{lem}[thm]{\protect\lemmaname}
  \theoremstyle{definition}
  \newtheorem{example}[thm]{\protect\examplename}

\date{}
\usepackage{multicol}
\usepackage{xypic}
\usepackage{etex}
\usepackage{amsthm}
\usepackage{fullpage,amsmath,amssymb,pictexwd}
\usepackage{graphics}
\usepackage{epsfig}
\usepackage{graphicx}
\usepackage{enumerate}
\usepackage{verbatim}
\usepackage{stmaryrd}
\usepackage{url}
\usepackage[pdfauthor={Andrew Wilfong},
            pdftitle={Toric Polynomial Generators of Complex Cobordism},
            pdfproducer={Latex with hyperref},
            pdfcreator={latex->dvips->ps2pdf},
            pdfpagemode=UseOutlines,
            bookmarksopen=true,
            letterpaper,
            bookmarksnumbered=true]{hyperref}
\hypersetup{bookmarksopenlevel=0, colorlinks=true, linkcolor=blue, citecolor=red}

\makeatother

\usepackage{babel}
  \providecommand{\conjecturename}{Conjecture}
  \providecommand{\corollaryname}{Corollary}
  \providecommand{\definitionname}{Definition}
  \providecommand{\examplename}{Example}
  \providecommand{\lemmaname}{Lemma}
  \providecommand{\propositionname}{Proposition}
  \providecommand{\remarkname}{Remark}
\providecommand{\theoremname}{Theorem}

\begin{document}

\title{Toric Polynomial Generators of Complex Cobordism}

\author{Andrew Wilfong}

\date{\today}
\begin{abstract}
Although it is well-known that the complex cobordism ring is a polynomial
ring $\Omega_{*}^{U}\cong\mathbb{Z}\left[\alpha_{1},\alpha_{2},\ldots\right]$,
an explicit description for convenient generators $\alpha_{1},\alpha_{2},\ldots$
has proven to be quite elusive. The focus of the following is to construct
complex cobordism polynomial generators in many dimensions using smooth
projective toric varieties. These generators are very convenient objects
since they are smooth connected algebraic varieties with an underlying
combinatorial structure that aids in various computations. By applying
certain torus-equivariant blow-ups to a special class of smooth projective
toric varieties, such generators can be constructed in every complex
dimension that is odd or one less than a prime power. A large amount
of evidence suggests that smooth projective toric varieties can serve
as polynomial generators in the remaining dimensions as well.
\end{abstract}



\maketitle

\section{Introduction}

In 1960, Milnor and Novikov independently showed that the complex
cobordism ring $\Omega_{*}^{U}$ is isomorphic to the polynomial ring
$\mathbb{Z}\left[\alpha_{1},\alpha_{2},\ldots\right]$, where $\alpha_{n}$
has complex dimension $n$ \cite{Novikov1960,Thom1995}. The standard
method for choosing generators $\alpha_{n}$ involves taking products
and disjoint unions of complex projective spaces and Milnor hypersurfaces
$\mathcal{H}_{i,j}\subset\mathbb{C}P^{i}\times\mathbb{C}P^{j}$. This
method provides a smooth algebraic \emph{not necessarily connected
}variety in each even real dimension whose cobordism class can be
chosen as a polynomial generator of $\Omega_{*}^{U}$. Replacing the
disjoint unions with connected sums give other choices for polynomial
generators. However, the operation of connected sum does not preserve
algebraicity, so this operation results in a smooth connected \emph{not
necessarily algebraic }manifold as a complex cobordism generator in
each dimension.

Since Milnor and Novikov's original construction, others have searched
for more convenient choices for generators of $\Omega_{*}^{U}$. For
example, Buchstaber and Ray provided an alternate construction of
polynomial generators in 1998 \cite{Buchstaber1998,Buchstaber2001}.
They described certain smooth projective toric varieties which multiplicatively
generate $\Omega_{*}^{U}$. As a consequence, disjoint unions of these
toric varieties can be chosen as polynomial generators. Taking connected
sums instead allows one to choose a convenient topological generalization
of a toric variety called a \emph{quasitoric manifold }as a generator
in each dimension. The advantage of these quasitoric generators is
that they have a convenient combinatorial structure that aids in many
computations. However, this technique still only provides examples
of generators that are connected or algebraic, but not both in general.

Several years later, Johnston took a drastically different approach
to constructing polynomial generators of complex cobordism which resulted
in the discovery of generators that are simultaneously connected and
algebraic \cite{Johnston2004}. More specifically, Johnston's construction
involves taking a sequence of blow-ups of hypersurfaces and complete
intersections in smooth projective algebraic varieties, starting with
complex projective space. By tracking the change of a certain cobordism
invariant called the Milnor genus, Johnston proved that every complex
cobordism polynomial ring generator $\alpha_{n}$ can be represented
by a smooth projective connected variety.

The purpose of the following is to apply techniques similar to those
of Johnston to search for even more convenient choices for complex
cobordism polynomial generators, namely smooth projective \emph{toric
}varieties. Not only are these connected and algebraic like Johnston's
generators, but they also display the computationally-convenient combinatorial
characteristics of Buchstaber and Ray's quasitoric generators.
\begin{conjecture}
\label{maincon}For each $n\ge1$, there exists a smooth projective
toric variety whose cobordism class can be chosen for the polynomial
generator $\alpha_{n}$ of $\Omega_{*}^{U}\cong\mathbb{Z}\left[\alpha_{1},\alpha_{2},\ldots\right]$.
\end{conjecture}
Taking torus-equivariant blow-ups of certain smooth projective toric
varieties will provide examples of such generators in most dimensions.
More specifically,
\begin{thm}
If $n$ is odd or $n$ is one less than a power of a prime, then the
cobordism class of a smooth projective toric variety can be chosen
for the complex cobordism ring polynomial generator of complex dimension
$n$.
\end{thm}
It seems very likely that generators can be found in the remaining
even dimensions as well using a similar strategy. In fact, this would
be a consequence of a certain number theory conjecture. Although this
conjecture has not yet been verified, there is a significant amount
of numerical evidence that supports it.
\begin{thm}
If $n\le100\,001$, then the cobordism class of a smooth projective
toric variety can be chosen for the complex cobordism ring polynomial
generator of complex dimension $n$.
\end{thm}
To prove these results, it is of course essential to know when a manifold
can be chosen to represent a polynomial generator of the complex cobordism
ring. Detecting polynomial generators of $\Omega_{*}^{U}$ involves
computing the value of a certain cobordism invariant.
\begin{defn}
Consider a stably complex manifold $M^{2n}$, and formally write it
Chern class as $c\left(M\right)=\prod\limits _{k=1}^{n}\left(1+x_{k}\right)$.
The \emph{Milnor genus }$s_{n}\left[M\right]$ of $M$ is the characteristic
number obtained by evaluating the cohomology class $s_{n}\left(c\left(M\right)\right)=\sum\limits _{k=1}^{n}x_{k}^{n}$
on the fundamental class of $M$, i.e.
\[
s_{n}\left[M\right]=\left\langle \sum_{k=1}^{n}x_{k}^{n},\mu_{M}\right\rangle \in\mathbb{Z}.
\]

\end{defn}
Milnor and Novikov proved that $\left[M^{2n}\right]$ can be chosen
for the polynomial generator $\alpha_{n}$ of $\Omega_{*}^{U}\cong\mathbb{Z}\left[\alpha_{1},\alpha_{2},\ldots\right]$
if and only if the following relation holds:
\begin{equation}
s_{n}\left[M^{2n}\right]=\begin{cases}
\pm1 & \mbox{ if }n+1\ne p^{m}\mbox{ for any prime }p\mbox{ and integer }m\\
\pm p & \mbox{ if }n+1=p^{m}\mbox{ for some prime }p\mbox{ and integer }m
\end{cases}\label{eq:gen}
\end{equation}
(see \cite{Stong1968} for details).

The focus of this paper is to construct smooth projective toric varieties
whose Milnor genera have the appropriate value in order for the variety
to be chosen for the polynomial generators $\alpha_{n}$ of complex
cobordism. Section 2 offers a brief introduction to toric varieties
and their pertinent topological properties. It also includes the construction
of certain smooth projective toric varieties $Y_{n}^{\varepsilon}\left(a,b\right)$
which are used in later sections to construct complex cobordism polynomial
generators. Section 3 proves the existence of smooth projective toric
variety polynomial generators in even complex dimensions one less
than a prime power. In Section 4, such generators are found in all
odd dimensions. In Section 5, the remaining unproven dimensions are
discussed. More specifically, a number-theoretic conjecture is presented
which is sufficient to verify the existence of smooth projective toric
variety polynomial generators in the remaining dimensions. Overwhelming
numerical evidence is given in support of this conjecture.

The established methods of Milnor, Novikov, Buchstaber, and Ray for
producing complex cobordism polynomial generators do not provide an
explicit universal description of generators, as their methods rely
on solving certain Diophantine equations. The techniques in this paper
and those of Johnston \cite{Johnston2004} still do not provide this
desirable universal description in most dimensions, since the constructions
involve finding a sequence of blow-ups of unspecified length. Section
6 discusses the possibility of finding a convenient, explicit description
of complex cobordism polynomial generators among smooth projective
toric varieties.

\section{Toric Varieties}

A \emph{toric variety} is a normal variety that contains the torus
as a dense open subset such that the action of the torus on itself
extends to an action on the entire variety. Remarkably, these varieties
are in one-to-one correspondence with objects from convex geometry
called fans. Therefore, studying the combinatorial properties of these
fans can reveal a great deal of information about the corresponding
toric varieties. See \cite{Fulton1993,Cox2011} for a more in-depth
treatment of toric varieties.
\begin{defn}
A \emph{(strongly convex rational polyhedral) cone }$\sigma$ spanned
by \emph{generating rays} $v_{1},\ldots,v_{m}\in\mathbb{Z}^{n}$ is
a set of points
\[
\sigma=\mbox{pos}\left(v_{1},\ldots,v_{m}\right)=\left\{ \sum\limits _{k=1}^{m}a_{k}v_{k}\in\mathbb{R}^{n}\Bigr|a_{k}\ge0\right\}
\]
such that $\sigma$ does not contain any lines passing through the
origin.

A \emph{fan }$\Delta$ in $\mathbb{R}^{n}$ is a set of cones in $\mathbb{R}^{n}$
such that each face of a cone in $\Delta$ also belongs to $\Delta$,
and the intersection of any two cones in $\Delta$ is a face of both
cones.

The one-dimensional cones of a fan are called its \emph{generating
rays}.
\end{defn}
A cone can be used to construct a $\mathbb{C}$-algebra which is the
coordinate ring of an affine toric variety. A fan can in turn be used
to construct an abstract toric variety. More specifically, if two
cones $\sigma_{1}$ and $\sigma_{2}$ of a fan intersect at a face
$\tau$, then the affine varieties $U_{\sigma_{1}}$ and $U_{\sigma_{2}}$
of the two cones can be glued together along the subvariety $U_{\tau}$
associated to $\tau$ to produce a toric variety associated to the
fan $\sigma_{1}\cup\sigma_{2}$. This construction demonstrates that
every fan defines a corresponding toric variety. In fact, the converse
is also true.
\begin{thm}
(\cite[Section 3.1]{Cox2011}) There is a bijective correspondence
between equivalence classes of fans in $\mathbb{R}^{n}$ under unimodular
transformations and isomorphism classes of complex $n$-dimensional
toric varieties.
\end{thm}
The fan corresponding to a variety $X$ will be denoted $\Delta_{X}$,
and the variety corresponding to a fan $\Delta$ will be denoted $X_{\Delta}$.
This bijection can be proven by examining the orbits of a toric variety
under the torus action. There is a bijective correspondence between
these orbits and the cones of the associated fan.
\begin{thm}
(\cite[Section 3.2]{Cox2011}) Consider a fan $\Delta$ in $\mathbb{R}^{n}$
and its associated complex dimension $n$ toric variety $X_{\Delta}$.
Every orbit of the torus action on $X_{\Delta}$ corresponds to a
distinct cone in $\Delta$. If such an orbit is a $k$-dimensional
torus, then the corresponding cone will have dimension $n-k$.
\end{thm}
As a result of this correspondence between fans and toric varieties,
many of the algebraic properties of toric varieties directly correspond
to properties of the associated fans.
\begin{prop}
(\cite{Kleinschmidt1988}) Consider a fan $\Delta$ in $\mathbb{R}^{n}$.

The toric variety $X_{\Delta}$ is compact if and only if $\Delta$
is a complete fan, i.e. the union of all of the cones in $\Delta$
is $\mathbb{R}^{n}$ itself.

The variety is smooth if and only if $\Delta$ is regular, i.e. every
maximal $n$-dimensional cone is spanned by $n$ generating rays that
form an integer basis.

The variety $X_{\Delta}$ is isomorphic to the variety $X_{\Delta'}$
if and only if there is a unimodular transformation $\mathbb{Z}^{n}\rightarrow\mathbb{Z}^{n}$
which maps $\Delta$ into $\Delta'$ and preserves the simplicial
structure of the fans.

The variety $X_{\Delta}$ is projective if and only if $\Delta$ is
normal to a lattice polytope (see \cite[Section 5.1]{Buchstaber2002}
for details about polytopes and their relation to toric varieties).

\end{prop}
The convenient combinatorial structure of a fan can also be used to
determine many important topological properties of the corresponding
toric varieties. For example, Jurkiewicz computed the integral cohomology
ring of a smooth projective toric variety, and Danilov generalized
the result to all smooth toric varieties.

Consider a complete regular fan $\Delta$ in $\mathbb{R}^{n}$ with
generating rays $v_{1},\ldots,v_{m}$. Each of the rays $v_{k}$ is
a one-dimensional cone in $\Delta$ which corresponds to a codimension
two subvariety $X_{k}$ of $X_{\Delta}$. Each of these subvarieties
determines a cohomology class in $H^{2}\left(X_{\Delta}\right)$ by
taking the image of the fundamental class $\left[X_{k}\right]$ of
$X_{k}$ under the composition
\[
H_{2n-2}\left(X_{k}\right)\hookrightarrow H_{2n-2}\left(X_{\Delta}\right)\rightarrow H^{2}\left(X_{\Delta}\right),
\]
where the first map is induced from inclusion and the second is Poincar\'e
duality. Denote the cohomology class in $H^{2}\left(X_{\Delta}\right)$
corresponding to the ray $v_{k}$ by $v_{k}$ as well. It will be
clear from context what the meaning of $v_{k}$ is.
\begin{thm}
\label{cohomSPTV}(\cite{Jurkiewicz1985,Danilov1978}) Suppose the
generating rays $v_{1},\ldots,v_{m}$ of a complete regular fan $\Delta$
in $\mathbb{R}^{n}$ are given by $v_{j}=\left(\lambda_{1j},\ldots,\lambda_{nj}\right)$.
For $i=1,\ldots,n$, set
\[
\theta_{i}=\lambda_{i1}v_{1}+\ldots+\lambda_{im}v_{m}\in\mathbb{Z}\left[v_{1},\ldots,v_{m}\right].
\]
Define $L=\left(\theta_{1},\ldots,\theta_{n}\right)$ to be the ideal
generated by these linear polynomials. Also define $J$ to be the
ideal generated by all square-free monomials $v_{i_{1}}\cdots v_{i_{k}}$
such that $v_{i_{1}},\ldots,v_{i_{k}}$ do \emph{not} span a cone
in $\Delta$ (the Stanley-Reisner ideal of $\Delta$). Then the integral
cohomology of the toric variety $X_{\Delta}$ is given by
\[
H^{*}\left(X_{\Delta}\right)\cong\mathbb{Z}\left[v_{1},\ldots,v_{m}\right]/\left(L+J\right).
\]

\end{thm}
The Chern class of a smooth toric variety can also be computed using
combinatorial data. The natural complex structure of a smooth toric
variety leads to a stable splitting of its tangent bundle, and this
splitting is encoded in the fan associated to the toric variety.
\begin{thm}
\label{ChernSPTV}(see \cite[Section 5.3]{Buchstaber2002} for details)
Given a complete regular fan $\Delta$ in $\mathbb{R}^{n}$ with generating
rays $v_{1},\ldots,v_{m}$, the total Chern class of $X_{\Delta}$
is given by
\[
c\left(X_{\Delta}\right)=\left(1+v_{1}\right)\left(1+v_{2}\right)\cdots\left(1+v_{m}\right)\in H^{*}\left(X_{\Delta}\right).
\]

\end{thm}
This splitting of the Chern class leads to a description of the Milnor
genus of a smooth toric variety in terms of its fan.
\begin{cor}
\label{MilnorSPTV}Let $X_{\Delta}$ be a smooth toric variety corresponding
to a complete regular fan $\Delta$ in $\mathbb{R}^{n}$ with generating
rays $v_{1},\ldots,v_{m}$. Then the Milnor genus of $X_{\Delta}$
is given by
\[
s_{n}\left[X_{\Delta}\right]=\left\langle \sum_{k=1}^{m}v_{k}^{n},\mu_{X_{\Delta}}\right\rangle .
\]

\end{cor}
Unfortunately, this formula is usually difficult to evaluate in most
cohomology rings of smooth toric varieties. The following proposition
is particularly useful in attempting these evaluations of characteristic
numbers.
\begin{prop}
\label{charicSPTV}(\cite[Section 5.1]{Fulton1993}) Suppose $\mbox{pos}\left(v_{1},\ldots,v_{n}\right)$
is a maximal cone of a complete regular fan $\Delta$ in $\mathbb{R}^{n}$.
Then evaluating $v_{1}\cdots v_{n}\in H^{2n}\left(X_{\Delta}\right)$
on the fundamental class $\mu_{X_{\Delta}}$ of the variety yields
one, i.e.
\[
\left\langle v_{1}\cdots v_{n},\mu_{X_{\Delta}}\right\rangle =1.
\]

\end{prop}
The blow-up $\mbox{Bl}_{V}X$ of a variety $X$ along a subvariety
$V$ can also be described using fans in the case of toric varieties
(see \cite[Chapter 1 Section 4 and Chapter 4 Section 6]{Griffiths1978}
for details about blow-ups). Consider a complete regular fan $\Delta$
in $\mathbb{R}^{n}$ containing a cone $\sigma$ of dimension $k$.
Then there are $k$-many generating rays $v_{1},\ldots,v_{k}$ of
$\Delta$ such that $\sigma=\mbox{pos}\left(v_{1},\ldots,v_{k}\right)$.
Construct a new fan $\mbox{Bl}_{\sigma}\Delta$ by first introducing
a new generating ray $x=v_{1}+\ldots+v_{k}$. To obtain the cones
of $\mbox{Bl}_{\sigma}\Delta$, first keep all cones in $\Delta$
that do not contain $\sigma$. Any cone $\tau$ in $\Delta$ that
contains $\sigma$ is no longer one of the cones in $\mbox{Bl}_{\sigma}\Delta$.
These cones $\tau$ in $\Delta$ of the form $\tau=\mbox{pos}\left(v_{1},\ldots,v_{k},v_{i_{1}},\ldots,v_{i_{j}}\right)$
are removed from $\mbox{Bl}_{\sigma}\Delta$ and replaced with all
cones of the form $\mbox{pos}\left(v_{1},\ldots,\hat{v}_{l},\ldots,v_{k},x,v_{i_{1}},\ldots,v_{i_{j}}\right)$.
That is, one of the rays of $\sigma$ is removed and replaced with
$x$ to obtain a new cone in $\mbox{Bl}_{\sigma}\Delta$. The fan
$\mbox{Bl}_{\sigma}\Delta$ is called the \emph{star subdivision }of
$\Delta$ relative to $\sigma$ (see \cite[Section 3.3]{Cox2011}
for details).
\begin{prop}
(\cite[Section 3.3]{Cox2011})  Let $\Delta$ be a complete regular
fan in $\mathbb{R}^{n}$. Consider a $k$-dimensional cone $\sigma=\mbox{pos}\left(v_{1},\ldots,v_{k}\right)$
in $\Delta$, and let $X_{\sigma}$ denote the $\left(n-k\right)$-dimensional
toric subvariety of $X_{\Delta}$ which is associated to the cone
$\sigma$. Then $X_{\mbox{\emph{Bl}}_{\sigma}\Delta}=\mbox{\emph{Bl}}_{X_{\sigma}}X_{\Delta}$.
That is, the blow-up of $X_{\Delta}$ along the subvariety $X_{\sigma}$
is a toric variety whose associated fan is the star subdivision of
$\Delta$ relative to $\sigma$.
\end{prop}
The operation of blowing up along torus-equivariant subvarieties preserves
several key properties of toric varieties. The following proposition
is well-known.
\begin{prop}
\label{blowupsmooth}The blow-up of a smooth projective toric variety
along a subvariety that is an orbit of the torus action is itself
a smooth projective toric variety.
\end{prop}
It is straight-forward to verify that the blow-up is smooth by computing
determinants of the maximal cones resulting from the star subdivision.
The fan of the blown up variety is normal to a polytope obtained by
truncating the polytope associated to the original variety along the
face corresponding to the cone being blown up. The resulting polytope
has vertices with rational coefficients. Dilating this polytope produces
a lattice polytope, so the blown up variety, whose fan is normal to
this polytope, is also projective.

\bigskip{}

In general, the complexity of the cohomology ring makes it challenging
to compute the Milnor genus of a smooth toric variety using Corollary
\ref{MilnorSPTV}. However, by carefully choosing toric varieties
with a convenient bundle\emph{ }structure and taking certain blow-ups,
one obtains a collection of smooth projective toric varieties that
are simple enough to allow their Milnor genera to be computed yet
still complicated enough to produce a wide array of possible values
for these Milnor genera. These varieties can be used to find complex
cobordism polynomial generators in most dimensions, and it seems likely
that they can be used as generators in every dimension.
\begin{defn}
Fix a complex dimension $n\ge3$, an integer $\varepsilon\in\left\{ 2,\ldots,n-1\right\} $,
and two integers $a$ and $b$. Define $U=\left\{ u_{1},\ldots,u_{n-\varepsilon+1}\right\} $,
where $u_{k}=e_{k}$ is the standard basis vector in $\mathbb{R}^{n}$
for $k=1,\ldots,n-\varepsilon$, and set $u_{n-\varepsilon+1}=\left(-1,\stackrel{\left(n-\varepsilon\right)}{\ldots},-1,0,\ldots,0\right)$.
Also define $V=\left\{ v_{1},\ldots,v_{\varepsilon}\right\} $, where
$v_{k}=e_{n-\varepsilon+k}$ is the standard basis vector for $k=1,\ldots,\varepsilon-1$,
and $v_{\varepsilon}=\left(0,\stackrel{\left(n-\varepsilon-1\right)}{\ldots\ldots},0,a,-1,\ldots,-1,0\right)$.
Finally, define $W=\left\{ w_{1},w_{2}\right\} $, where $w_{1}=e_{n}$
and $w_{2}=\left(0,\ldots,0,b,-1\right)$. A fan $\Delta_{n}^{\varepsilon}\left(a,b\right)$
in $\mathbb{R}^{n}$ can be defined by using the $\left(n+3\right)$-many
generating rays in $U\cup V\cup W$. A maximal cone in $\Delta_{n}^{\varepsilon}\left(a,b\right)$
is obtained by choosing for generators $\left(n-\varepsilon\right)$-many
vectors from $U$, $\left(\varepsilon-1\right)$-many vectors from
$V$, and $1$ vector from $W$. Let $Y_{n}^{\varepsilon}\left(a,b\right)$
denote the toric variety corresponding to this fan.
\end{defn}
It is easy to verify that $\Delta_{n}^{\varepsilon}\left(a,b\right)$
is a complete regular fan that is normal to a lattice polytope. Therefore,
$Y_{n}^{\varepsilon}\left(a,b\right)$ is a compact smooth projective
toric variety. More specifically, $\Delta_{n}^{\varepsilon}\left(a,b\right)$
can be viewed as the join of three separate fans $\Delta_{U}$, $\Delta_{V}$,
and $\Delta_{W}$ whose generating rays belong to $U$, $V$, and
$W$, respectively (see Figure \ref{fig:fan}). On the level of toric
varieties, $Y_{n}^{\varepsilon}\left(a,b\right)$ is a stack of two
projectivized bundles. More specifically, the toric variety corresponding
to $\Delta_{V}*\Delta_{W}$ is a $\mathbb{C}P^{\varepsilon-1}$-bundle
over $X_{\Delta_{W}}\cong\mathbb{C}P^{1}$. The variety $Y_{n}^{\varepsilon}\left(a,b\right)$
is a $\mathbb{C}P^{n-\varepsilon}$-bundle over the variety corresponding
to $\Delta_{V}*\Delta_{W}$. Refer to \cite[Section 3.3]{Cox2011}
for more details about obtaining fiber bundle structures from fans
such as these.

\begin{figure}
\begin{center}\includegraphics{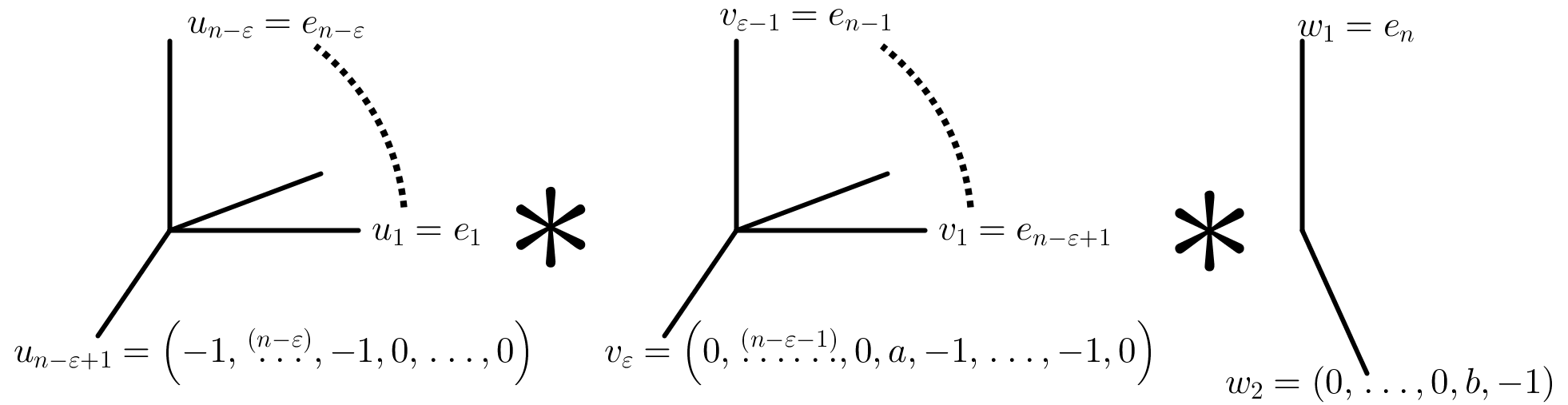}\end{center}

\caption{$\Delta_{n}^{\varepsilon}\left(a,b\right)$ depicted as a join of
fans\label{fig:fan}}
\end{figure}

The bundle structure of $Y_{n}^{\varepsilon}\left(a,b\right)$ makes
it convenient to calculate its cohomology ring and Milnor genus.
\begin{prop}
\label{sn}Fix a complex dimension $n$, an integer $\varepsilon\in\left\{ 2,\ldots,n-1\right\} $
and two integers $a$ and $b$. Define $R_{n}\left(\varepsilon\right)=n-\varepsilon+\left(-1\right)^{\varepsilon}{n-1 \choose \varepsilon}$.
The Milnor genus of $Y_{n}^{\varepsilon}\left(a,b\right)$ is given
by
\[
s_{n}\left[Y_{n}^{\varepsilon}\left(a,b\right)\right]=a^{\varepsilon}bR_{n}\left(\varepsilon\right).
\]
\end{prop}
\begin{proof}
By Theorem \ref{cohomSPTV},
\[
H^{*}\left(Y_{n}^{\varepsilon}\left(a,b\right)\right)\cong\mathbb{Z}\left[u_{1},\ldots,u_{n-\varepsilon+1},v_{1},\ldots,v_{\varepsilon},w_{1},w_{2}\right]/\left(L+J\right)
\]
where
\begin{align*}
L & =\bigl(u_{1}-u_{n-\varepsilon+1},\ldots,u_{n-\varepsilon-1}-u_{n-\varepsilon+1},u_{n-\varepsilon}-u_{n-\varepsilon+1}+av_{\varepsilon},\\
 & \ \ \ \ \ \ \ v_{1}-v_{\varepsilon},\ldots,v_{\varepsilon-2}-v_{\varepsilon},v_{\varepsilon-1}-v_{\varepsilon}+bw_{2},w_{1}-w_{2}\bigl)
\end{align*}
and $J=\left(u_{1}\cdots u_{n-\varepsilon+1},v_{1}\cdots v_{\varepsilon},w_{1}\cdot w_{2}\right)$.
Let $u,v,w\in H^{*}\left(Y_{n}^{\varepsilon}\left(a,b\right)\right)$
denote the cohomology classes corresponding to the generating rays
$u_{n-\varepsilon+1},\ v_{\varepsilon}$, and $w_{2}$, respectively.
Then the cohomology ring of $Y_{n}^{\varepsilon}\left(a,b\right)$
simplifies to become
\[
H^{*}\left(Y_{n}^{\varepsilon}\left(a,b\right)\right)\cong\mathbb{Z}\left[u,v,w\right]/\left(u^{n-\varepsilon+1}-au^{n-\varepsilon}v,v^{\varepsilon}-bv^{\varepsilon-1}w,w^{2}\right).
\]
The Milnor genus of $Y_{n}^{\varepsilon}\left(a,b\right)$ can be
computed by first evaluating
\[
s_{n}\left(c\left(Y_{n}^{\varepsilon}\left(a,b\right)\right)\right)=u_{1}^{n}+\ldots+u_{n-\varepsilon+1}^{n}+v_{1}^{n}+\ldots+v_{\varepsilon}^{n}+w_{1}^{n}+w_{2}^{n}
\]
in this ring. Doing so yields
\[
s_{n}\left(c\left(Y_{n}^{\varepsilon}\left(a,b\right)\right)\right)=a^{\varepsilon}b\left(n-\varepsilon+\left(-1\right)^{\varepsilon}{n-1 \choose \varepsilon}\right)u^{n-\varepsilon}v^{\varepsilon-1}w.
\]
Since $\mbox{pos}\left(u_{1},\ldots,u_{n-\varepsilon-1},u_{n-\varepsilon+1},v_{1},\ldots,v_{\varepsilon-2},v_{\varepsilon},w_{2}\right)$
is a maximal cone in $\Delta_{n}^{\varepsilon}\left(a,b\right)$ and
we have the relation $u_{1}\cdots u_{n-\varepsilon-1}\cdot u_{n-\varepsilon+1}\cdot v_{1}\cdots v_{\varepsilon-2}\cdot v_{\varepsilon}\cdot w_{2}=u^{n-\varepsilon}v^{\varepsilon-1}w$
in $H^{*}\left(Y_{n}^{\varepsilon}\left(a,b\right)\right)$, we have
$\left\langle u^{n-\varepsilon}v^{\varepsilon-1}w,\mu_{Y_{n}^{\varepsilon}\left(a,b\right)}\right\rangle =1$
by Proposition \ref{charicSPTV}. Then
\[
s_{n}\left[Y_{n}^{\varepsilon}\left(a,b\right)\right]=a^{\varepsilon}b\left(n-\varepsilon+\left(-1\right)^{\varepsilon}{n-1 \choose \varepsilon}\right)=a^{\varepsilon}bR_{n}\left(\varepsilon\right).
\]

\end{proof}
As will be seen in the next section, the smooth projective toric varieties
$Y_{n}^{\varepsilon}\left(a,b\right)$ provide examples of polynomial
ring generators of $\Omega_{*}^{U}$ in a limited number of dimensions.
To obtain examples of toric variety generators in more dimensions,
we can apply certain blow-ups to these varieties. The most basic and
useful of these blow-ups is the blow-up at a torus-fixed point. It
is straight-forward to calculate the change in Milnor genus during
this operation.
\begin{prop}
\label{blowupMil}(cf. \cite[Lemma 3.4]{Johnston2004}) Consider a
complex manifold $M^{2n}$ and its blow-up $Bl_{x}M$ at $x\in M$.
The change in Milnor genus is given by the following formula.
\[
s_{n}\left[Bl_{x}M\right]=\begin{cases}
s_{n}\left[M\right]-\left(n+1\right) & \mbox{ if }n\mbox{ is even}\\
s_{n}\left[M\right]-\left(n-1\right) & \mbox{ if }n\mbox{ is odd}
\end{cases}
\]
\end{prop}
\begin{proof}
This formula is a consequence of the well-known fact that $Bl_{x}M$
is diffeomorphic to $M\#\overline{\mathbb{C}P^{n}}$ as an oriented
differentiable manifold, where $\overline{\mathbb{C}P^{n}}$ is the
complex projective space with the opposite of the standard orientation
(see \cite[Proposition 2.5.8]{Huybrechts2005} for details). We can
compute $s_{n}\left[\overline{\mathbb{C}P^{n}}\right]=-\left(n+\left(-1\right)^{n}\right)$,
which gives the desired formula.
\end{proof}

\section{Toric Polynomial Generators in Some Even Dimensions}

The smooth projective toric varieties $Y_{n}^{\varepsilon}\left(a,b\right)$
provide examples of polynomial generators of $\Omega_{*}^{U}$ in
certain dimensions. For example, the following theorem is an immediate
consequence of Proposition \ref{sn}.
\begin{thm}
\label{p-1}If $n=p-1$ for some prime $p\ge5$, then the smooth projective
toric variety $Y_{n}^{n-2}\left(1,1\right)$ can be chosen to represent
the generator $\alpha_{n}$ of $\Omega_{*}^{U}\cong\mathbb{Z}\left[\alpha_{1},\alpha_{2},\ldots\right]$.\end{thm}
\begin{rem}
\label{CP^n}There are likely to be a wide array of different smooth
projective toric varieties that can be chosen as polynomial generators.
For example, if $n=p-1$ for some prime $p$, then $s_{n}\left[\mathbb{C}P^{n}\right]=p$.
Thus the simpler toric variety $\mathbb{C}P^{n}$ can be chosen to
represent the generator $\alpha_{n}$ of $\Omega_{*}^{U}\cong\mathbb{Z}\left[\alpha_{1},\alpha_{2},\ldots\right]$.
In fact, it is easy to show that $\mathbb{C}P^{n}$ is not cobordant
to $Y_{n}^{n-2}\left(1,1\right)$, so there are at least two distinct
toric polynomial generators in dimensions that are one less than a
prime $p\ge5$.
\end{rem}
Theorem \ref{p-1} can be generalized to dimensions one less than
a power of an odd prime by examining blow-ups of the $Y_{n}^{\varepsilon}\left(a,b\right)$.
In this situation, a cobordism class must have Milnor genus $\pm p$
for it to be used as a polynomial generator in the complex cobordism
ring (see \eqref{eq:gen}). Recall that each blow-up at a point in
this even complex dimension decreases the Milnor genus by $n+1=p^{m}$
by Proposition \ref{blowupMil}. This means that in order to find
a smooth projective toric variety with Milnor genus $p$, it suffices
to find one whose Milnor genus is positive and is congruent to $p$
modulo $p^{m}$. The extra multiples of $p^{m}$ can then be removed
by a sequence of blow-ups at points. By choosing these points to be
torus-fixed points, each successive blow-up is itself a smooth projective
toric variety.

A technical lemma is needed to show that some of the $Y_{n}^{\varepsilon}\left(a,b\right)$
satisfy the desired congruence in these dimensions.
\begin{lem}
\label{num}Let $n=p^{m}-1$ for some odd prime $p$ and integer $m\ge2$.
Then
\[
R_{n}\left(p^{m-1}\right)\equiv-p\mbox{ mod }p^{m}\mbox{ and }R_{n}\left(p^{m-1}\right)<0,
\]
where as in Proposition \ref{sn}, $R_{n}\left(p^{m-1}\right)=n-p^{m-1}-{n-1 \choose p^{m-1}}$.\end{lem}
\begin{proof}
To prove that $R_{n}\left(p^{m-1}\right)\equiv-p\mbox{ mod }p^{m}$,
first consider ${n-1 \choose p^{m-1}}\mbox{ mod }p^{m}$. We can write
\begin{align}
{n-1 \choose p^{m-1}} & ={p^{m}-2 \choose p^{m-1}}\nonumber \\
{n-1 \choose p^{m-1}} & =\frac{p^{m}-2}{2}\cdot\frac{p^{m}-3}{3}\cdots\frac{p^{m}-p^{m-1}}{p^{m-1}}\cdot\left(p^{m}-\left(p^{m-1}+1\right)\right).\label{eq:precancel}
\end{align}

In general, if $p\not|c$, then $c$ has a multiplicative inverse
$c^{-1}$ in the multiplicative group of integers $\mathbb{Z}_{p^{m}}^{\times}$.
In this situation,
\[
\frac{p^{m}-c}{c}=c^{-1}\left(p^{m}-c\right)\equiv-1\mbox{ mod }p^{m}.
\]
If $p|c$, then this cancellation cannot be applied. Applying these
cancellations to \eqref{eq:precancel} yields
\begin{align*}
{n-1 \choose p^{m-1}} & \equiv\frac{p^{m}-p}{p}\cdot\frac{p^{m}-2p}{2p}\cdots\frac{p^{m}-p^{m-1}}{p^{m-1}}\cdot\left(p^{m-1}+1\right)\mbox{ mod }p^{m}\\
 & \equiv\frac{p^{m-1}-1}{1}\cdot\frac{p^{m-1}-2}{2}\cdots\frac{p^{m-1}-p^{m-2}}{p^{m-2}}\cdot\left(p^{m-1}+1\right)\mbox{ mod }p^{m}.
\end{align*}
Applying this same cancellation procedure repeatedly eventually produces
\begin{align*}
{n-1 \choose p^{m-1}} & \equiv\frac{p^{2}-1}{1}\cdot\frac{p^{2}-2}{2}\cdots\frac{p^{2}-p}{p}\cdot\left(p^{m-1}+1\right)\mbox{ mod }p^{m}\\
 & \equiv p-p^{m-1}-1\mbox{ mod }p^{m}.
\end{align*}
Then $R_{n}\left(p^{m-1}\right)\equiv p^{m}-1-p^{m-1}-\left(p-p^{m-1}-1\right)\equiv-p\mbox{ mod }p^{m}$.

To see that $R_{n}\left(p^{m-1}\right)$ is negative, note that for
any integer $n\ge3^{2}-1=8$ (the smallest possible dimension for
this lemma), $n<{n-1 \choose 2}$. Then given any $p^{m}$ where $p$
is prime and $m\ge2$, $n=p^{m}-1<{p^{m}-2 \choose 2}<{p^{m}-2 \choose p^{m-1}}$.
Then $R_{n}\left(p^{m-1}\right)=p^{m}-1-p^{m-1}-{p^{m}-2 \choose p^{m-1}}<0$. \end{proof}
\begin{thm}
\label{p^m-1}If $n=p^{m}-1$ for some odd prime $p$ and some integer
$m\ge2$, then there exists a smooth projective toric variety whose
cobordism class can be chosen for the polynomial generator $\alpha_{n}$
of $\Omega_{*}^{U}=\mathbb{Z}\left[\alpha_{1},\alpha_{2},\ldots\right]$. \end{thm}
\begin{proof}
Consider the smooth projective toric variety $Y_{n}^{p^{m-1}}\left(1,-1\right)$.
By Proposition \ref{sn} and Lemma \ref{num},
\[
s_{n}\left[Y_{n}^{p^{m-1}}\left(1,-1\right)\right]=-1\cdot R_{n}\left(p^{m-1}\right)\equiv p\mbox{ mod }p^{m}\mbox{ and }s_{n}\left[Y_{n}^{p^{m-1}}\left(1,-1\right)\right]\ge p.
\]
Since each blow-up at a point decreases the Milnor genus by $n+1=p^{m}$
(by Proposition \ref{blowupMil}), applying sufficiently many blow-ups
to torus-fixed points of $Y_{n}^{p^{m-1}}\left(1,-1\right)$ will
produce a smooth projective toric variety with Milnor genus $p$.
The cobordism class of this variety can be used as a polynomial generator
of $\Omega_{*}^{U}$ by \eqref{eq:gen}.\end{proof}
\begin{example}
Suppose $n=5^{2}-1=24$. Then
\[
s_{24}\left[Y_{24}^{5}\left(1,-1\right)\right]=-R_{24}\left(5\right)=33630\equiv5\mbox{ mod }25.
\]
Each blow-up of a point in this dimension decreases the Milnor genus
by $n+1=25$. By applying a sequence of $1345$ many blow-ups at torus-fixed
points to $Y_{24}^{5}\left(1,-1\right)$, one obtains a smooth projective
toric variety with Milnor genus $5$. The cobordism class of this
variety can be used as the polynomial generator $\alpha_{24}$ of
$\Omega_{*}^{U}=\mathbb{Z}\left[\alpha_{1},\alpha_{2},\ldots\right]$
by \eqref{eq:gen}.
\end{example}
This example demonstrates that although Theorem \ref{p^m-1} verifies
the existence of smooth projective toric variety polynomial generators
in certain dimensions, the theorem is not very useful in explicitly
constructing such examples.

\section{Toric Polynomial Generators in Odd Dimensions}

A limited number of odd-dimensional generators can be chosen from
the $Y_{n}^{\varepsilon}\left(a,b\right)$ themselves. The following
theorem is a direct consequence of Proposition \ref{sn}.
\begin{thm}
\label{2^m-1}If $n=2^{m}-1$ for some integer $m\ge2$, then the
smooth projective toric variety $Y_{n}^{n-1}\left(1,1\right)$ can
be chosen to represent the generator $\alpha_{n}$ of $\Omega_{U}^{*}\cong\mathbb{Z}\left[\alpha_{1},\alpha_{2},\ldots\right]$.
\end{thm}
Smooth projective toric variety cobordism generators can be obtained
in the remaining odd dimensions by considering certain blow-ups. First,
a simple number theory fact is needed.
\begin{lem}
\label{oddnum}Let $n$ be a positive odd integer. If $n\ne2^{k}-1$
for any $k\in\mathbb{Z}$, then
\[
n\equiv2^{m}-1\mbox{ mod }2^{m+1}
\]
for some integer $m\ge1$. \end{lem}
\begin{proof}
Suppose $n$ is odd and $n\ne2^{k}-1$ for any $k\in\mathbb{Z}$.
Then $n+1=2^{m}\cdot q$, where $m\ge1$, $q>1$, and $2\not|q$.
Then $n+1-2^{m}=2^{m}\left(q-1\right)$, and $q-1$ is even. Then
$2^{m+1}|\left(n+1-2^{m}\right)$, so $n\equiv2^{m}-1\mbox{ mod }2^{m+1}$.
\end{proof}
In order to obtain smooth projective toric variety polynomial generators
of $\Omega_{*}^{U}$ in the remaining odd dimensions, we must first
blow up a particular two-dimensional subvariety of $Y_{n}^{\varepsilon}\left(a,b\right)$.
The change in Milnor genus during this blow-up can be determined.
\begin{lem}
\label{blowup}Fix an odd complex dimension $n\ge3$. Let $a$, $b$,
and $\varepsilon$ be arbitrary integers such that $\varepsilon\in\left\{ 2,\ldots,n-1\right\} $.
Consider the cone $\sigma=\mbox{pos}\left(u_{1},\ldots,u_{n-\varepsilon},v_{1},\ldots,v_{\varepsilon-1}\right)$
in $\Delta_{n}^{\varepsilon}\left(a,b\right)$ of dimension $n-1$.
This cone corresponds to a real dimension two subvariety $X_{\sigma}$
of $Y_{n}^{\varepsilon}\left(a,b\right)$. If $Y_{n}^{\varepsilon}\left(a,b\right)$
is blown up along $X_{\sigma}$, then the Milnor genus of the resulting
smooth projective toric variety $\mbox{\emph{Bl}}_{X_{\sigma}}Y_{n}^{\varepsilon}\left(a,b\right)$
is given by
\[
s_{n}\left[\mbox{\emph{Bl}}_{X_{\sigma}}Y_{n}^{\varepsilon}\left(a,b\right)\right]=s_{n}\left[Y_{n}^{\varepsilon}\left(a,b\right)\right]+2b.
\]
\end{lem}
\begin{proof}
Let $x=\left(1,\ldots,1,0\right)$ be the additional generating ray
obtained when finding the star subdivision of $\Delta_{n}^{\varepsilon}\left(a,b\right)$
relative to $\sigma$. By Theorem \ref{cohomSPTV},
\[
H^{*}\left(\mbox{Bl}_{X_{\sigma}}Y_{n}^{\varepsilon}\left(a,b\right)\right)\cong\mathbb{Z}\left[u_{1},\ldots,u_{n-\varepsilon+1},v_{1},\ldots,v_{\varepsilon},w_{1},w_{2},x\right]/\left(L+J\right)
\]
where
\begin{align*}
L & =\bigl(u_{1}-u_{n-\varepsilon+1}+x,\ldots,u_{n-\varepsilon-1}-u_{n-\varepsilon+1}+x,u_{n-\varepsilon}-u_{n-\varepsilon+1}+av+x,\\
 & \ \ \ \ \ \ \ v_{1}-v_{\varepsilon}+x,\ldots,v_{\varepsilon-2}-v_{\varepsilon}+x,v_{\varepsilon-1}-v_{\varepsilon}+bw+x,w_{1}-w_{2}\bigl)
\end{align*}
and
\[
J=\left(u_{1}\cdots u_{n-\varepsilon+1},v_{1}\cdots v_{\varepsilon},w_{1}\cdot w_{2},u_{1}\cdots u_{n-\varepsilon}\cdot v_{1}\cdots v_{\varepsilon-1},u_{n-\varepsilon+1}\cdot x,v_{\varepsilon}\cdot x\right).
\]
Let $u,v,w,x\in H^{*}\left(\mbox{Bl}_{X_{\sigma}}Y_{n}^{\varepsilon}\left(a,b\right)\right)$
denote the cohomology classes corresponding to the generating rays
$u_{n-\varepsilon+1},\ v_{\varepsilon}$, $w_{2}$, and $x$, respectively.
Then the cohomology ring of $\mbox{Bl}_{X_{\sigma}}Y_{n}^{\varepsilon}\left(a,b\right)$
simplifies to become
\[
H^{*}\left(\mbox{Bl}_{X_{\sigma}}Y_{n}^{\varepsilon}\left(a,b\right)\right)\cong\mathbb{Z}\left[u,v,w,x\right]/I,
\]
where
\begin{align*}
I & =\bigl(ux,vx,w^{2},u^{n-\varepsilon+1}-au^{n-\varepsilon}v,v^{\varepsilon}-bv^{\varepsilon-1}w,\\
 & \ \ \ \ \ \ \ u^{n-\varepsilon}v^{\varepsilon-1}-bu^{n-\varepsilon}v^{\varepsilon-2}w+bwx^{n-2}+x^{n-1}\bigl).
\end{align*}
The Milnor genus of $\mbox{Bl}_{X_{\sigma}}Y_{n}^{\varepsilon}\left(a,b\right)$
can be computed by first evaluating
\[
s_{n}\left(c\left(\mbox{Bl}_{X_{\sigma}}Y_{n}^{\varepsilon}\left(a,b\right)\right)\right)=u_{1}^{n}+\ldots+u_{n-\varepsilon+1}^{n}+v_{1}^{n}+\ldots+v_{\varepsilon}^{n}+w_{1}^{n}+w_{2}^{n}+x^{n}
\]
in this ring. Doing so yields
\[
s_{n}\left(c\left(\mbox{Bl}_{X_{\sigma}}Y_{n}^{\varepsilon}\left(a,b\right)\right)\right)=a^{\varepsilon}b\left(n-\varepsilon+\left(-1\right)^{\varepsilon}{n-1 \choose \varepsilon}\right)u^{n-\varepsilon}v^{\varepsilon-1}w+2bu^{n-\varepsilon}v^{\varepsilon-1}w.
\]
Since $\mbox{pos}\left(u_{1},\ldots,u_{n-\varepsilon},v_{1},\ldots,v_{\varepsilon-2},w_{2},x\right)$
is a maximal cone in $\mbox{Bl}_{X_{\sigma}}\Delta_{n}^{\varepsilon}\left(a,b\right)$
and we have the relation $u_{1}\cdots u_{n-\varepsilon}\cdot v_{1}\cdots v_{\varepsilon-2}\cdot w_{2}\cdot x=u^{n-\varepsilon}v^{\varepsilon-1}w$
in $H^{*}\left(\mbox{Bl}_{X_{\sigma}}Y_{n}^{\varepsilon}\left(a,b\right)\right)$,
\[
\left\langle u^{n-\varepsilon}v^{\varepsilon-1}w,\mu_{\mbox{Bl}_{X_{\sigma}}Y_{n}^{\varepsilon}\left(a,b\right)}\right\rangle =1
\]
by Proposition \ref{charicSPTV}. Then
\[
s_{n}\left[\mbox{Bl}_{X_{\sigma}}Y_{n}^{\varepsilon}\left(a,b\right)\right]=a^{\varepsilon}b\left(n-\varepsilon+\left(-1\right)^{\varepsilon}{n-1 \choose \varepsilon}\right)+2b=s_{n}\left[Y_{n}^{\varepsilon}\left(a,b\right)\right]+2b.
\]
\end{proof}
\begin{thm}
\label{odd}If $n$ is odd, then there exists a smooth projective
toric variety whose cobordism class can be chosen as the polynomial
generator $\alpha_{n}$ of $\Omega_{*}^{U}\cong\mathbb{Z}\left[\alpha_{1},\alpha_{2},\ldots\right]$.\end{thm}
\begin{proof}
For $n=1$, use $\alpha_{1}=\left[\mathbb{C}P^{1}\right]$. If $n=2^{m}-1$
for some $m\ge2$, then we can choose $\alpha_{n}=\left[Y_{n}^{n-1}\left(1,1\right)\right]$
by Theorem \ref{2^m-1}. Now assume that $n\ne2^{k}-1$ for any integer
$k$. Then by Lemma \ref{oddnum}, there exists an integer $m\ge1$
such that $n\equiv\left(2^{m}-1\right)\mbox{ mod }2^{m+1}$.

In this situation, a smooth projective toric variety can be constructed
that is congruent to $1\mbox{ mod }n-1$. In order to find this variety,
first consider $R_{n}\left(2^{m}\right)=n-2^{m}+{n-1 \choose 2^{m}}$.
Since $n-1\equiv\left(2^{m}-2\right)\mbox{ mod }2^{m+1}$ and $n\ne2^{m}-1$,
we have $n-1=2^{m+1}K+2^{m}-2$ for some positive integer $K$. Let
$K=2^{i}+2^{i+1}K_{i+1}+2^{i+2}K_{i+2}+\ldots$ be the binary expansion
of $K$, where $i$ is the minimum index with a nonzero coefficient.
Note that the coefficient of $2$ is zero in the binary expansion
$2^{m+1}K=2^{m+i+1}+2^{m+i+2}K_{i+1}+\ldots$ since $m\ge1$. Then
\[
2^{m+1}K-2=2+2^{2}+\ldots+2^{m+i}+2^{m+i+1}\cdot0+2^{m+i+2}\cdot K_{i+1}+\ldots.
\]
The coefficient of $2^{m}$ in this binary expansion is one regardless
of the value of $i$. Then the coefficient of $2^{m}$ in the binary
expansion of $2^{m+1}K+2^{m}-2$ is zero. Then by Lucas's Theorem,
\begin{align*}
{n-1 \choose 2^{m}} & ={2^{m+1}K+2^{m}-2 \choose 2^{m}}\\
 & \equiv{0 \choose 0}{1 \choose 0}\cdots{1 \choose 0}{0 \choose 1}{1 \choose 0}\cdots{1 \choose 0}{0 \choose 0}{K_{i+1} \choose 0}{K_{i+2} \choose 0}\cdots\mbox{ mod }2,
\end{align*}
where ${0 \choose 1}$ is the factor corresponding to the coefficients
of $2^{m}$ in $n-1$ and $2^{m}$. Since this factor is zero, ${n-1 \choose 2^{m}}\equiv0\mbox{ mod }2$.
Then $R_{n}\left(2^{m}\right)\equiv n-2^{m}\mbox{ mod }2$, i.e. $R_{n}\left(2^{m}\right)$
is odd since $n$ is odd.

Next consider the integer $n-1$, and let $p_{1},\ldots,p_{k}$ be
its odd prime factors. Set $a=p_{1}\cdots p_{k}$. If $n-1$ has no
odd prime factors, then set $a=1$. Each $p_{i}$ divides $a^{2^{m}}R_{n}\left(2^{m}\right)$,
so none of the $p_{i}$ divide $a^{2^{m}}R_{n}\left(2^{m}\right)+2$.
Since it is also odd, $a^{2^{m}}R_{n}\left(2^{m}\right)+2$ is an
element of $\mathbb{Z}_{n-1}^{\times}$, the multiplicative group
of integers modulo $n-1$. Choose an integer $b$ to represent its
inverse in $\mathbb{Z}_{n-1}^{\times}$, and choose the sign of $b$
to guarantee that $b\cdot\left(a^{2^{m}}R_{n}\left(2^{m}\right)+2\right)>0$.
Then $ba^{2^{m}}R_{n}\left(2^{m}\right)+2b\equiv1\mbox{ mod }n-1$.
By Proposition \ref{sn} and Lemma \ref{blowup},
\[
s_{n}\left[\mbox{Bl}_{X_{\sigma}}Y_{n}^{2^{m}}\left(a,b\right)\right]=ba^{2^{m}}R_{n}\left(2^{m}\right)+2b\equiv1\mbox{ mod }n-1,
\]
where $\sigma=\mbox{pos}\left(u_{1},\ldots,u_{n-2^{m}},v_{1},\ldots,v_{2^{m}-1}\right)$.
By Proposition \ref{blowupMil}, applying sufficiently many blow-ups
to torus-fixed points of the smooth projective toric variety $\mbox{Bl}_{X_{\sigma}}Y_{n}^{2^{m}}\left(a,b\right)$
will eventually produce a smooth projective toric variety with Milnor
genus one. This variety can be chosen to represent the cobordism polynomial
ring generator $\alpha_{n}$. \end{proof}
\begin{example}
Suppose $n=43$. Then $n\equiv\left(2^{2}-1\right)\mbox{ mod }2^{3}$,
so we use $m=2$ to get $R_{43}\left(2^{2}\right)=111969$. The integer
$n-1=42$ has odd prime factors $3$ and $7$, so we must set $a=3\cdot7=21$.
Then $21^{4}\cdot R_{43}\left(4\right)+2=21775843091$. The inverse
of $21775843091$ in $\mathbb{Z}_{42}^{\times}$ can be represented
by $b=11$. Consider the smooth projective toric variety $\mbox{Bl}_{X_{\sigma}}Y_{43}^{4}\left(21,11\right)$,
where $\sigma=\mbox{pos}\left(u_{1},\ldots,u_{39},v_{1},\ldots,v_{3}\right)$.
Using Proposition \ref{sn} and Lemma \ref{blowup}, its Milnor genus
is
\[
s_{43}\left[\mbox{Bl}_{X_{\sigma}}Y_{43}^{4}\left(21,11\right)\right]=239534274001\equiv1\mbox{ mod }42.
\]
In this dimension, each blow-up of a torus-fixed point decreases the
Milnor genus by $42$. Applying a sequence of $5703197000$ blow-ups
of torus-fixed points of $\mbox{Bl}_{X_{\sigma}}Y_{43}^{4}\left(21,11\right)$
produces a smooth projective toric variety with Milnor genus $239534274001-42\cdot5703197000=1$.
This smooth projective toric variety can be chosen to represent the
complex cobordism polynomial ring generator $\alpha_{43}$.

This example demonstrates that once again, these techniques are only
useful in establishing the existence of smooth projective toric variety
polynomial generators in certain dimensions. The actual varieties
that are obtained are still not convenient to work with.
\end{example}

\section{Toric Polynomial Generators in the Remaining Even Dimensions}

Smooth projective toric variety polynomial generators of the complex
cobordism ring have now been found in many dimensions. More specifically,
the cobordism class of a smooth projective toric variety can be chosen
as the polynomial generator $\alpha_{n}$ of $\Omega_{*}^{U}\cong\mathbb{Z}\left[\alpha_{1},\alpha_{2},\ldots\right]$
for any dimension $n$ such that $n$ is odd or $n$ is one less than
a power of a prime (see Theorems \ref{p-1}, \ref{p^m-1}, \ref{2^m-1}, and \ref{odd}).
The only dimensions in which smooth projective toric variety cobordism
polynomial generators have not yet been constructed are those for
which $n$ is even and $n+1$ is not a prime power. While a proof
of the conjecture in these dimensions remains elusive, there is overwhelming
numerical evidence that suggests that the conjecture is true. In fact,
it appears that a similar technique could be used to find toric polynomial
generators in these remaining dimensions. If a certain number-theoretic
result holds, then smooth projective toric varieties could be constructed
in a way to guarantee that a sequence of blow-ups at torus-fixed points
produces a smooth projective toric variety cobordism polynomial generator.
\begin{conjecture}
\label{numcon}Suppose $n$ is even and $n+1$ is not a prime power.
Then there exists an integer $\varepsilon\in\left\{ 2,\ldots,n-1\right\} $
such that $\gcd\left(R_{n}\left(\varepsilon\right),n+1\right)=1$.
\end{conjecture}
Suppose this conjecture is true. Given a complex even dimension $n$
such that $n+1$ is not a prime power, choose $\varepsilon$ to satisfy
the conjecture. Choose an integer $b$ to represent the inverse of
$R_{n}\left(\varepsilon\right)$ in $\mathbb{Z}_{n+1}^{\times}$,
and choose the sign of $b$ so that $bR_{n}\left(\varepsilon\right)>0$.
Then by Proposition \ref{sn}, $s_{n}\left[Y_{n}^{\varepsilon}\left(1,b\right)\right]=bR_{n}\left(\varepsilon\right)\equiv1\mbox{ mod }n+1$.
By Proposition \ref{blowupMil}, each blow-up at a torus-fixed point
in this dimension decreases the Milnor genus by $n+1$. Applying a
sequence of such blow-ups to $Y_{n}^{\varepsilon}\left(1,b\right)$
will eventually produce a smooth projective toric variety with Milnor
genus equal to one. By \eqref{eq:gen}, this variety can be chosen
to represent the cobordism polynomial ring generator $\alpha_{n}$.

A simple computer program can be used to verify this conjecture in
relatively low dimensions.
\begin{prop}
\label{numcomp}Suppose $n$ is even and $n+1$ is not a prime power.
If $n\le100\,000$ then there exists an integer $\varepsilon\in\left\{ 2,\ldots,n-1\right\} $
such that $\gcd\left(R_{n}\left(\varepsilon\right),n+1\right)=1$.\end{prop}
\begin{rem}
If $n\ne20$ and $n\ne50$, then an $\varepsilon$ that is prime and
greater than the largest prime factor of $n+1$ can be chosen to satisfy
Proposition \ref{numcomp}. For $n=20$ and $n=50$, we can choose
$\varepsilon=7$ and $\varepsilon=21$, respectively. A much faster
and more efficient computer program can be used to verify Conjecture
\ref{numcon} in the remaining dimensions $n\le100\,000$ by only
checking prime numbers for $\varepsilon$. \end{rem}
\begin{cor}
If $n\le100\,001$, then there exists a smooth projective toric variety
whose cobordism class can be chosen for the polynomial generator $\alpha_{n}$
of $\Omega_{*}^{U}\cong\mathbb{Z}\left[\alpha_{1},\alpha_{2},\ldots\right]$.
\end{cor}
Not only is there an integer $\varepsilon$ satisfying Conjecture
\ref{numcon} in dimensions $n\le100\,000$, but the number of such
$\varepsilon$ seems to increase in general as $n$ increases. Figure
\ref{fig:graph} displays this trend. It shows the number of $\varepsilon$
satisfying Conjecture \ref{numcon} for each even $n\le10\,000$ such
that $n+1$ is not a prime power. In order to verify the conjecture,
only one such $\varepsilon$ needs to exist for any given $n$. It
seems likely that the trend in the graph would continue for larger
$n$, making it doubtful that there exists some large complex dimension
$n$ for which there is no corresponding $\varepsilon$ that satisfies
the conjecture.

\begin{figure}
\begin{center}\includegraphics[scale=0.65]{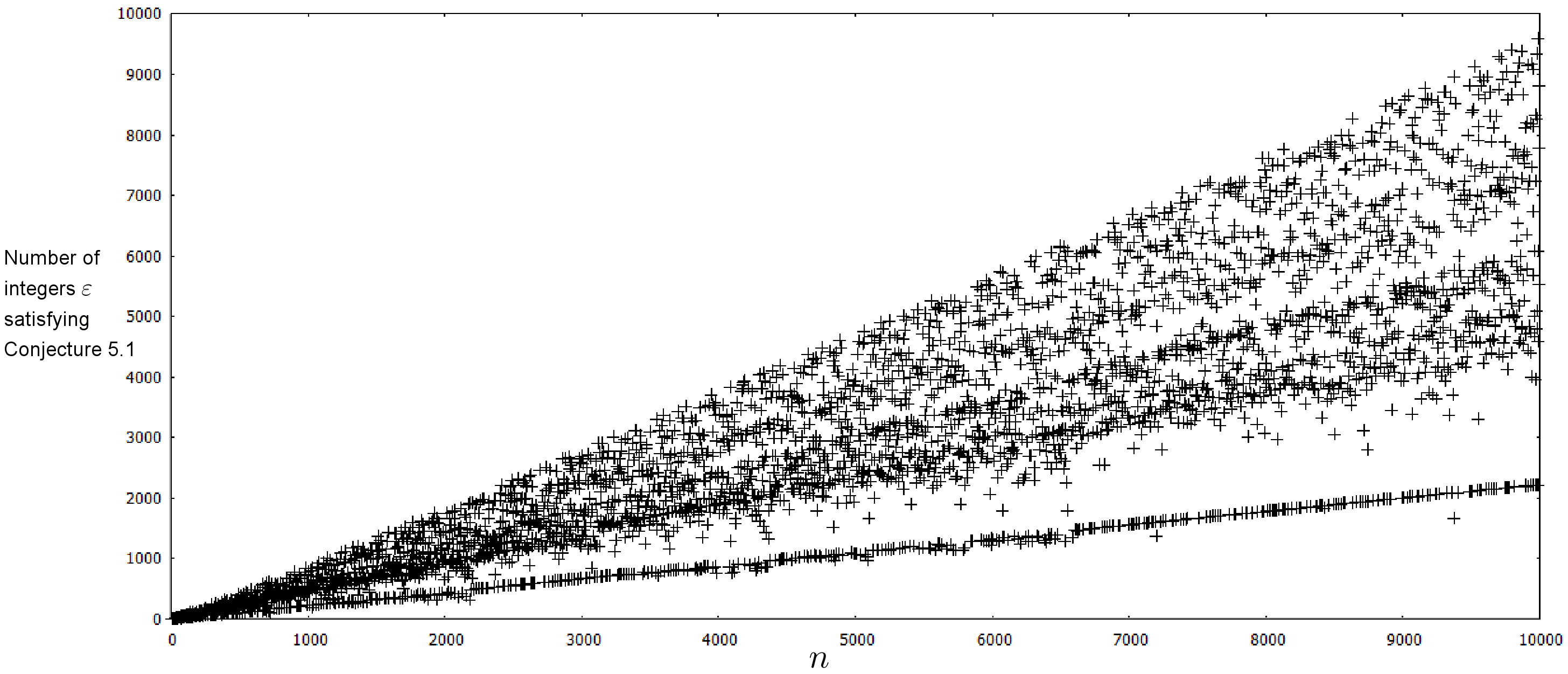}\end{center}

\caption{The number of integers $\varepsilon$ satisfying Conjecture \ref{numcon}
for $n$ up to $10\,000$\label{fig:graph}}
\end{figure}

\section{Conclusion}

The evidence supporting Conjecture \ref{numcon} makes it seem very
likely that a smooth projective toric variety can be chosen to represent
the polynomial generators of the complex cobordism ring in each dimension.
Finding a proof of Conjecture \ref{numcon} may be the easiest way
to verify this. Unfortunately, the techniques that have been used
to prove the existence of smooth projective toric variety polynomial
generators still do not result in very convenient choices (see Theorems
\ref{p^m-1} and \ref{odd} and also Conjecture \ref{numcon}).

Remark \ref{CP^n} and Figure \ref{fig:graph} suggest that there
may be many non-cobordant choices for smooth projective toric variety
polynomial generators in a given dimension. It therefore seems worthwhile
to search for other smooth projective toric varieties for which, like
the $Y_{n}^{\varepsilon}\left(a,b\right)$, the Milnor genus is straight-forward
to compute, and there is a large variety of possible values for these
Milnor genera. Perhaps this would lead to the discovery of smooth
projective toric varieties that can be chosen as polynomial generator
representatives that are also easy to describe and work with.

Recall that the varieties $Y_{n}^{\varepsilon}\left(a,b\right)$ consist
of a stack of two $\mathbb{C}P^{i}$-bundles over some $\mathbb{C}P^{k}$.
As an example of the possible diversity of toric polynomial generators,
we could instead consider certain smooth projective toric varieties
classified by Kleinschmidt that can be viewed as $\mathbb{C}P^{k}$-bundles
over $\mathbb{C}P^{n-k}$ \cite{Kleinschmidt1988}. At the level of
fans, these varieties correspond to fans which have exactly two more
generating rays than the dimension. These provide additional, often
non-cobordant examples of smooth projective toric variety polynomial
generators in many dimensions \cite[Chapter 5]{Wilfong2013}.

There are many other examples of smooth projective toric varieties
that may also be useful in finding complex cobordism polynomial generators.
For example, Batyrev classified all smooth projective toric varieties
corresponding to fans with three more generating rays than the dimension
\cite{Batyrev1991}. These display a convenient structure which facilitates
computations of Milnor genera. Cayley polytopes (see \cite{Dickenstein2009}
for details) also display a simple structure which facilitates computing
the Milnor genus of the corresponding toric varieties. More refined
techniques for computing the Milnor genera of these smooth projective
toric varieties could lead to the discovery of convenient and easy
to describe complex cobordism polynomial generators among them.

\bibliographystyle{spmpsci}
\bibliography{Bibliography}

\end{document}